\documentclass[11pt]{amsart}
\usepackage{amsmath,amsfonts,amscd,amssymb}
\usepackage[margin=3.4cm]{geometry}
\theoremstyle{plain}
\newcounter{thmcount}

\newtheorem{theorem}[thmcount]{Theorem}

\newtheorem{lemma}[thmcount]{Lemma}
\newtheorem{corollary}[thmcount]{Corollary}
\theoremstyle{definition}
\newtheorem{remark}[thmcount]{Remark}
\newtheorem*{remark*}{Remark}

\newtheorem{definition}[thmcount]{Definition}

\usepackage[utf8]{inputenc}
\usepackage{amssymb}
\usepackage[english]{babel}

\usepackage[OT2,T1]{fontenc}
\DeclareSymbolFont{cyrletters}{OT2}{wncyr}{m}{n}
\DeclareMathSymbol{\Sha}{\mathalpha}{cyrletters}{"58}

\usepackage{amsthm}
\usepackage{mathtools}
\usepackage{dirtytalk}
\usepackage[alphabetic,lite]{amsrefs}

\usepackage{hyperref}

\mathtoolsset{showonlyrefs}

\title{$p^{\infty}$--Selmer ranks of CM Abelian Varieties}

\subjclass[2020]{11G15, 14K22}

\author{Jamie Bell}
\address{University College London, Gower Street, London, WC1E 6BT, UK}
\email{james.bell.20@ucl.ac.uk}

\begin{document}

\maketitle

\begin{abstract}
    For an elliptic curve with complex multiplication over a number field, the $p^{\infty}$--Selmer rank is even for all $p$. \v{C}esnavi\v{c}ius proved this using the fact that $E$ admits a $p$-isogeny whenever $p$ splits in the complex multiplication field, and invoking known cases of the $p$-parity conjecture. We give a direct proof, and generalise the result to abelian varieties.
\end{abstract}

\section*{Introduction}
An elliptic curve over a number field with endomorphisms other than multiplication by an integer is said to have complex multiplication. These can make the curves easier to work with. It is well known, and important for this paper, that they always have even rank. Indeed, suppose we have an elliptic curve $E$ over $K$ with complex multiplication, and note that $E(K) \otimes_{\mathbb{Z}} \mathbb{Q}$ is a vector space over $\mathbb{Q}$ with dimension equal to the rank of $E$. Now the endomorphisms must in fact form a lattice in an imaginary quadratic field $L$. Then $E(K) \otimes_{\mathbb{Z}} \mathbb{Q}$ is an $L$--vector space, and so an even-dimensional $\mathbb{Q}$--vector space.

It is not hard to prove that these curves also have root number 1 (\cite{kestutis}, Prop. 6.3), so they satisfy the parity conjecture. We might hope that there is an analogous result for Selmer groups. This is in fact true, however it is not so easy to prove. In this paper we will present a new proof of this result, and generalise it to abelian varieties (Theorem \ref{maintheorem}). To see why the analogous proof fails, note that we are now looking at $\mathbb{Q}_p$--vector spaces. Suppose $p=5$ and $L=\mathbb{Q}(i)$. Then we can have an action of $L$ on a one-dimensional $\mathbb{Q}_p$--vector space, because $i \in \mathbb{Q}_5$. The method works when $p$ is inert in $L$, but not when it splits.

Suppose we have an elliptic curve $E$ over a number field $K$, which has complex multiplication. Looking at the Tate--Shafarevich group $\Sha$, we find its $p$-primary part is isomorphic to $(\textrm{finite group}) \times (\mathbb{Q}_p/\mathbb{Z}_p)^{\delta_p}$ for some integer $\delta_p$.
We will define the $p^{\infty}$--Selmer rank to be $\mathrm{rk}_p(E)=\mathrm{rk}(E)+\delta_p$.

A generalisation of elliptic curves is abelian varieties. The aim of this paper is to prove the following:

\begin{theorem}\label{maintheorem}
Suppose $A/K$ is an abelian variety with complex multiplication by $M$, and $p$ a prime. Then $\mathrm{rk}_p(A)$ is even.
\end{theorem}

Given that the rank is even, Theorem \ref{maintheorem} is equivalent to the statement that the divisible part of $\Sha$ has even $\mathbb{Z}_p$-corank. This is in fact expected to be 0, as $\Sha$ is conjectured to be finite.

Another reason to expect Theorem \ref{maintheorem} to hold is the $p$-parity conjecture, which states that for an abelian variety $A$ over a number field $K$, with root number $w(A/K)$,
\begin{equation}
    (-1)^{\mathrm{rk}_p(A/K)}=w(A/K).
\end{equation}

In the CM case, the root number is 1 (\cite{kestutis}, Proposition 6.3). In a different form $p$-parity was conjectured by Selmer in 1954 (\cite{selmer}). The conjecture is known in the case where $A$ is an elliptic curve over a number field admitting a $p$-isogeny thanks to T. and V. Dokchitser (\cite{DD11}) and \v{C}esnavi\v{c}ius (\cite{kestutis}). By calculating root numbers, this allowed \v{C}esnavi\v{c}ius to conclude that for elliptic curves with complex multiplication, $\mathrm{rk}_p(A)$ is even. However there is no equivalent $p$-parity result to use for abelian varieties, so we must use a different method. 

Throughout the paper we use `complex multiplication' or `CM' to mean complex multiplication defined over $K$. With CM defined over $\bar{\mathbb{Q}}$, the $p$-parity conjecture has been proved for elliptic curves over totally real $K$, but is open in general. For $p\neq2$ this is due to Nekovar (\cite{Nek2}, 5.10) and for $p=2$ Green and Maistret (\cite{GM}, 6.5).

From Theorem \ref{maintheorem} we can deduce the following:

\begin{corollary}
    Suppose $A$ and $p$ are as in Theorem \ref{maintheorem}. If $\Sha[p^{\infty}]$ is infinite, then it contains $(\mathbb{Q}_p/\mathbb{Z}_p)^2$. 
\end{corollary}

\subsection*{Notation}
Throughout, we will assume $A$ and $B$ are abelian varieties over a number field $K$, with a polarisation $\lambda: A \rightarrow \hat{A}$ defined over $K$.

For an isogeny $f$ and a field $L$, denote by $f_{A(L)}$ the map on $L$-points induced by $f$, and similarly let $f_{\Sha}$ by the map induced on the Tate--Shafarevich group.

We will denote the dual of an abelian variety $A$ by $\hat{A}$, and of an isogeny $f: A \rightarrow B$ by $\hat{f}: \hat{B} \rightarrow \hat{A}$.

For a prime $p$, $\Sha[p^{\infty}]$ is the $p$-primary part of $\Sha$. $\delta_p$ will denote the multiplicity of $\mathbb{Q}_p/\mathbb{Z}_p$ in $\Sha[p^{\infty}]$. 
Specifically, $\Sha[p^{\infty}] \cong (\textrm{finite group}) \times (\mathbb{Q}_p/\mathbb{Z}_p)^{\delta_p}$ for some integer $\delta_p$. $\Sha_d$ will be the divisible part of $\Sha$ and $\Sha_{nd}$ the quotient of $\Sha$ by $\Sha_d$.

Write $Y_p(A/K)$ for $\mathrm{Hom}(\Sha_d[p^{\infty}],\mathbb{Q}_p/\mathbb{Z}_p)$, the Pontryagin dual of $\Sha_d[p^{\infty}]$. Let $\mathcal{Y}_p(A/K) = Y_p(A/K) \otimes_{\mathbb{Z}_p} \mathbb{Q}_p$. Note this is a $\mathbb{Q}_p$-vector space of dimension $\delta_p$, and an $\mathrm{End}_K(A) \otimes_{\mathbb{Z}} \mathbb{Q}_p$-module.

\begin{definition}[Rosati involution]
    For an abelian variety $A$ with polarisation $\lambda$, the Rosati involution is the involution on $\mathrm{End}_K(A) \otimes_{\mathbb{Z}} \mathbb{Q}$ sending $f$ to $f^{\dag} := \lambda^{-1} \circ \hat{f} \circ \lambda$. We extend this by continuity to $\mathrm{End}_K(A) \otimes_{\mathbb{Z}} \mathbb{Q}_p$.
\end{definition}

\begin{definition}[Complex multiplication]\label{CMdef}
    We say $A$ has complex multiplication if 
    \begin{itemize}
        \item $\mathrm{End}_K(A)\otimes\mathbb{Q} \supset M$, where $M$ is a totally complex field containing a totally real field $L$, $[M:L]=2$ and $[L:\mathbb{Q}]=\mathrm{dim}(A)$
        \item The Rosati involution corresponds to complex conjugation on $M$.
    \end{itemize}
\end{definition}
Note that this is complex multiplication over $K$, not over $\bar{K}$ (which is sometimes called potential complex multiplication).

\subsection*{Acknowledgements} I would like to thank my supervisor Vladimir Dokchitser for suggesting this problem, and for his advice and guidance. I would also like to thank the reviewers for their comments, and their help in strengthening the result.

This work was supported by the Engineering and Physical Sciences Research Council [EP/L015234/1], the EPSRC Centre for Doctoral Training in Geometry and Number Theory (The London School of Geometry and Number Theory) at University College London.

\section{Self-isogenies}
Suppose $A$ and $B$ are abelian varieties over a number field $K$, and $f:A \rightarrow B$ an isogeny between them.

Recall the following theorem.
\begin{theorem}[\cite{Milne}, Proof of I.7.3, I.7.3.1]\label{l3.1}
There is some finite set $S$ of places of $K$ such that
\begin{equation}
    \prod_{v \in S} \frac{\#\mathrm{ker}(f_{A(K_v)})}{\#\mathrm{coker}(f_{A(K_v)})} = \frac{\#\mathrm{ker}(f_{A(K)})}{\#\mathrm{coker}(f_{A(K)})} \cdot \frac{\#\mathrm{coker}(\hat{f}_{\hat{B}(K)})}{\#\mathrm{ker}(\hat{f}_{\hat{B}(K)})} \cdot \frac{\#\mathrm{ker}(\hat{f}_{\Sha})}{\#\mathrm{ker}(f_{\Sha})}.
\end{equation}
\end{theorem}

\begin{corollary}\label{cor6}
    Suppose that $A=B$, i.e. $f$ is a self-isogeny. Then 
    \begin{equation}
        \#\mathrm{ker}(\hat{f}_{\Sha})=\#\mathrm{ker}(f_{\Sha}).
    \end{equation}
\end{corollary}
\begin{proof}
    In the formula in theorem \ref{l3.1}, the left hand side is equal to the ratio of the volume forms of $A$ and $B$ which appear in the BSD coefficient (see \cite{Milne} Section I.7 for a full definition; in the notation of this chapter the volume term is $\frac{\prod_{\nu \in S}\mu_{\nu}(A,\omega)}{|\mu|^d}$). These depend only on $A$ and not $f$ so when $A=B$, this is 1.

Similarly, the next two terms equal the ratio of the regulators of $A$ and $B$ and the orders of their torsion subgroups. Specifically,
\begin{equation*}
    \frac{\#\mathrm{ker}(f_{A(K)})}{\#\mathrm{coker}(f_{A(K)})} \cdot \frac{\#\mathrm{coker}(\hat{f}_{\hat{B}(K)})}{\#\mathrm{ker}(\hat{f}_{\hat{B}(K)})}  = \frac{\mathrm{Reg}(A/K)\#B(K)_{tors}\#\hat{B}(K)_{tors}}{\mathrm{Reg}(B/K)\#A(K)_{tors}\#\hat{A}(K)_{tors}},
\end{equation*}
therefore when $A=B$ this is also 1.
\end{proof}
The following variant of this result will be useful.
\begin{lemma}
Suppose $f$ is as above, and $p$ any prime. Then
\begin{equation}
    \#\mathrm{ker}(\hat{f}_{\Sha[p^{\infty}]})=\#\mathrm{ker}(f_{\Sha[p^{\infty}]}),
\end{equation}
\end{lemma}
\begin{proof}
For any prime $p$, the $p$-adic valuations of the kernels in Lemma \ref{cor6} must be equal. As both kernels decompose as a product over primes $l$ of their $l$-primary subgroups, the $p$-part of each side comes from $\Sha[p^{\infty}]$, so we can replace $\Sha$ by $\Sha[p^{\infty}]$ and still have equality.
\end{proof}

\begin{lemma}\label{lemma17}
Suppose $f$ is as before. Then we can split the kernels into divisible and non-divisible parts. Specifically,
\begin{equation}
    \#\mathrm{ker}(\hat{f}_{\Sha_d})\#\mathrm{ker}(\hat{f}_{\Sha_{nd}})=\#\mathrm{ker}(f_{\Sha_d})\#\mathrm{ker}(f_{\Sha_{nd}}).
\end{equation}
\end{lemma}
\begin{proof}
We will show $\#\mathrm{ker}(f_{\Sha})=\#\mathrm{ker}(f_{\Sha_{d}})\#\mathrm{ker}(f_{\Sha_{nd}})$ and similarly for $\hat{f}$. This holds by an application of the snake lemma to the exact sequence
\begin{equation}
    0 \rightarrow \Sha_d \rightarrow \Sha \rightarrow \Sha_{nd} \rightarrow 0
\end{equation}
with the isogeny $f$, which is valid because $f$ maps $\Sha_d$ to $\Sha_d$. We can also see that $\#\mathrm{coker}(f_{\Sha_d})=1$, because $f$ has a conjugate isogeny $g: A \rightarrow A$. This has the property that $f \circ g = [\mathrm{deg}(f)]$, and multiplication by an integer is surjective on $\Sha_d$. The result follows.
\end{proof}

\begin{lemma}\label{lemma9}
Let $A/K$ be an abelian variety, $p$ a prime, and $f:A \rightarrow A$ an isogeny defined over $K$. Then
\begin{equation}
    \#\mathrm{ker}(\hat{f}_{\Sha_d[p^{\infty}]})=\#\mathrm{ker}(f_{\Sha_d[p^{\infty}]}).
\end{equation}
\end{lemma}
\begin{proof}
By the functoriality and non-degeneracy of the Cassels-Tate pairing on $\Sha_{nd}$, \begin{equation}
    \#\mathrm{ker}(\hat{f}_{\Sha_{nd}[p^{\infty}]})=\#\mathrm{coker}(f_{\Sha_{nd}[p^{\infty}]})
\end{equation} 
(\cite{Milne}, proof of I.7.3).
Now $\#\mathrm{coker}(f_{\Sha_{nd}[p^{\infty}]})$ and $\#\mathrm{ker}(f_{\Sha_{nd}[p^{\infty}]})$ are equal, because $\Sha_{nd}[p^{\infty}]$ is a finite group. So the non-divisible parts of the equation in Lemma \ref{lemma17} cancel out, and we have equality of the divisible parts.
\end{proof}

\section{Complex Multiplication}
\begin{lemma}\label{lemma10}
    Suppose $A/K$ is a polarised abelian variety, $p$ a prime, and $\phi$ an invertible element of $\mathrm{End}_K(A) \otimes_{\mathbb{Z}} \mathbb{Q}_p$. Then
    \begin{equation*}
        \mathrm{ord}_p \mathrm{det} (\phi_{\mathcal{Y}_p}) = \mathrm{ord}_p \mathrm{det} (\phi^{\dag}_{\mathcal{Y}_p}).
    \end{equation*}
\end{lemma}
\begin{proof}
    We prove this for $\phi$ an isogeny of $A$ defined over $K$, and the full result follows. By properties of Pontryagin duality,
    \begin{equation*}
        \#\mathrm{ker}(\phi_{\Sha_d[p^{\infty}]}) = \# \mathrm{coker}(\phi_{Y_p}).
    \end{equation*}
    Now $\phi_{Y_p}$ can be represented over $\mathbb{Z}_p$ by a matrix $P$ in Smith normal form, with all diagonal entries non-zero. Then
    \begin{equation*}
        \mathrm{ord}_p\mathrm{det}(\phi_{\mathcal{Y}_p}) = \mathrm{ord}_p\mathrm{det}(P) = \mathrm{ord}_p\# \mathrm{coker}(P) = \mathrm{ord}_p\# \mathrm{coker}(\phi_{Y_p}).
    \end{equation*}
    It therefore follows from Lemma \ref{lemma9} that 
    \begin{equation*}
        \mathrm{ord}_p\mathrm{det}(\phi_{\mathcal{Y}_p}) = \mathrm{ord}_p\mathrm{det}(\hat{\phi}_{\mathcal{Y}_p}).
    \end{equation*}
    Now as $\hat{\phi} = \lambda \circ \phi^{\dag} \circ \lambda^{-1}$, $\phi^{\dag}_{\mathcal{Y}_p}$ will have the same determinant, and the result follows.    
\end{proof}

Suppose from now on that $A$ has complex multiplication by a field $M$, with totally real subfield $L$. Now $\mathcal{Y}_p(A/K)$ is an $M \otimes_{\mathbb{Q}}\mathbb{Q}_p$-module. This is isomorphic to $\prod_{\mathfrak{p}|p} M_{\mathfrak{p}}$, where the product is over primes $\mathfrak{p}$ of $M$ lying above $p$, and $M_{\mathfrak{p}}$ is the completion of $M$ at $\mathfrak{p}$. We can therefore decompose $\mathcal{Y}_p(A/K)$ into a sum of $\mathbb{Q}_p$-vector spaces
\begin{equation*}
    \mathcal{Y}_p(A/K) = \bigoplus_{\mathfrak{p}|p}V_{\mathfrak{p}},
\end{equation*}
where each $V_{\mathfrak{p}}$ is an $M_{\mathfrak{p}}$-vector space.

\begin{lemma}\label{lemma11}
    For each prime $\mathfrak{p}|p$, we have 
    \begin{equation*}
        \mathrm{dim}_{\mathbb{Q}_p} V_{\mathfrak{p}} = \mathrm{dim}_{\mathbb{Q}_p} V_{\bar{\mathfrak{p}}}.
    \end{equation*}
\end{lemma}
\begin{proof}
    If $\mathfrak{p}=\bar{\mathfrak{p}}$, we are done, so suppose they are not equal. Then define $\alpha$ to be the element of $M \otimes_{\mathbb{Q}}\mathbb{Q}_p = \prod_{\mathfrak{p}|p} M_{\mathfrak{p}}$ which corresponds to $p$ in $M_{\mathfrak{p}}$ and 1 in all the other factors. Now we can view $\alpha$ as an element of $\mathrm{End}_K(A) \otimes_{\mathbb{Z}}\mathbb{Q}_p$. Then
    \begin{equation*}
        \mathrm{ord}_p\mathrm{det}(\alpha_{\mathcal{Y}_p(A/K)}) = \mathrm{ord}_p\mathrm{det}(p|V_{\mathfrak{p}}) = \mathrm{dim}_{\mathbb{Q}_p}V_{\mathfrak{p}}.
    \end{equation*}
    Now by the definition of complex multiplication, $\alpha^{\dag}$ acts as $\bar{\alpha}$. It therefore acts as the identity on $V_{\mathfrak{q}}$ for $\mathfrak{q} \neq \bar{\mathfrak{p}}$, and as multiplication by $p$ on $V_{\bar{\mathfrak{p}}}$. So by the same argument we have
    \begin{equation*}
        \mathrm{ord}_p\mathrm{det}(\alpha^{\dag}_{\mathcal{Y}_p(A/K)}) = \mathrm{dim}_{\mathbb{Q}_p}V_{\bar{\mathfrak{p}}},
    \end{equation*}
    and by Lemma \ref{lemma10} the result follows.
\end{proof}

\begin{proof}[Proof of Theorem 1]
     It suffices to show that $\mathrm{dim}_{\mathbb{Q}_p}\mathcal{Y}_p(A/K) = \sum_{\mathfrak{p}|p}\mathrm{dim}_{\mathbb{Q}_p}V_{\mathfrak{p}}$ is even. Let $L$ be the fixed field of complex conjugation on $M$. If $\mathfrak{p}$ is inert or ramified in $M/L$, then $[M_{\mathfrak{p}}:\mathbb{Q}_p]$ is even. Therefore $$\mathrm{dim}_{\mathbb{Q}_p}V_{\mathfrak{p}} = [M_{\mathfrak{p}}:\mathbb{Q}_p]\mathrm{dim}_{M_{\mathfrak{p}}}V_{\mathfrak{p}}$$
     is also even.

     For the primes $\mathfrak{p}$ which split in $M/L$, we have $\mathfrak{p} \neq \bar{\mathfrak{p}}$, and $$\mathrm{dim}_{\mathbb{Q}_p}V_{\mathfrak{p}} = \mathrm{dim}_{\mathbb{Q}_p}V_{\bar{\mathfrak{p}}}.$$

     Thus $\sum_{\mathfrak{p}|p}\mathrm{dim}_{\mathbb{Q}_p}V_{\mathfrak{p}}$ is even, and so is $\mathrm{rk}_p(A/K)$.
\end{proof}

\begin{remark}
    The complex multiplication assumption can be weakened. Suppose $A$ is an abelian variety with $\mathrm{End}_K(A) \otimes_{\mathbb{Z}} \mathbb{Q} \supset M$, for some field $M$, and suppose the Rosati involution induces a non-trivial automorphism on $M$. Then we can still show that $\mathrm{rk}_p(A)$ is even. Here denote this automorphism by $\phi \mapsto \bar{\phi}$, and the fixed field of it plays the role of $L$. Then the proof proceeds in the same way.
\end{remark}

\begin{remark}
    A similar argument can also be applied directly to $p^{\infty}$-Selmer groups instead of $\Sha_d$. Let $X_p(A/K) = \mathrm{Hom}(\mathrm{Sel}_{p^{\infty}}(A/K),\mathbb{Q}_p/\mathbb{Z}_p)$ and $\mathcal{X}_p(A/K) = X_p(A/K) \otimes_{\mathbb{Z}_p} \mathbb{Q}_p$. Note that the $\mathbb{Q}_p$-rank of $\mathcal{X}_p(A/K)$ is $\mathrm{rk}_p(A/K)$. Then replace Theorem \ref{l3.1}  with Theorem 7 from \cite{DD10}. This tells us that for any self-isogeny $\phi$, $Q(\phi) = Q(\hat{\phi})$, where, for an isogeny $\psi:A \rightarrow B$, $$Q(\psi) := |\mathrm{coker}(\psi: A(K)/A(K)_{\mathrm{tors}} \rightarrow B(K)/B(K)_{\mathrm{tors}})| \times |\mathrm{ker}(\psi: \Sha(A)_{\mathrm{div}} \rightarrow \Sha(B)_{\mathrm{div}})|.$$

    Section 2 of \cite{DD09} tells us that $$\mathrm{ord}_pQ(\phi) = \mathrm{ord}_p|\mathrm{coker}(\phi:X_p(A/K) \rightarrow X_p(A/K))|.$$ 
    By the same arguments as in the proof of Lemma \ref{lemma10}, we can show that for any invertible $\phi \in \mathrm{End}_K(A) \otimes_{\mathbb{Z}} \mathbb{Q}_p$, $$\mathrm{ord}_p\mathrm{det}(\phi_{\mathcal{X}_p}) = \mathrm{ord}_p|\mathrm{coker}(\phi_{X_p})| = \mathrm{ord}_p|\mathrm{coker}(\hat{\phi}_{X_p})| = \mathrm{ord}_p\mathrm{det}(\hat{\phi}_{\mathcal{X}_p}) = \mathrm{ord}_p\mathrm{det}(\phi^{\dag}_{\mathcal{X}_p}).$$
    Then by an argument similar to Lemma \ref{lemma11} and the surrounding discussion, $\mathrm{rk}_p(A/K) = \mathrm{dim}_{\mathbb{Q}_p}(\mathcal{X}_p(A/K))$ is even.
\end{remark}

\end{document}